\documentclass[12pt]{article}
\usepackage{amsthm}
\usepackage{amssymb}
\usepackage{amsmath}
\usepackage{graphicx}
\usepackage{varwidth}
\usepackage{algpseudocode}
\usepackage{color}
\usepackage{tikz}
\usepackage{verbatim}
\usetikzlibrary{arrows}
\newtheorem{thm}{Theorem}[section]

\newtheorem{lem}[thm]{Lemma}

\newtheorem{fact}[thm]{Fact}

\theoremstyle{definition}
\newtheorem{defn}[thm]{Definition}

\theoremstyle{remark}

\begin{document}

\title{On the Existence of $t$-Identifying Codes in Undirected De Bruijn Graphs}%
\author{Victoria Horan \thanks{\texttt{victoria.horan.1@us.af.mil}} \\  Air Force Research Laboratory \\ Information Directorate}

\date{\today}%

\maketitle

\let\thefootnote\relax\footnote{Approved for public release; distribution unlimited:  88ABW-2015-3796.}

\begin{abstract}
This paper proves the existence of $t$-identifying codes on the class of undirected de Bruijn graphs with string length $n$ and alphabet size $d$, referred to as $\mathcal{B}(d,n)$.  It is shown that $\mathcal{B}(d,n)$ is $t$-identifiable whenever:
\begin{itemize}
    \item  $d \geq 3$ and $n \geq 2t$, and $t \geq 1$.
    \item  $d \geq 3$, $n \geq 3$, and $t=2$.
    \item  $d = 2$, $n \geq 3$, and $t=1$.
\end{itemize}
The remaining cases remain open.  Additionally, we show that the eccentricity of the undirected non-binary de Bruijn graph is $n$.
\end{abstract}

\section{Introduction and Background}

Let $x \in V(G)$, and define the ball of radius $t$ to be the set of all vertices $y$ with $d(x,y) \leq t$.  The formal definition of a $t$-identifying code on a graph $G$ is as follows.

\begin{defn}
    A subset $S \subseteq V(G)$ is a $t$-identifying code in a graph $G$ if the following conditions are met.
    \begin{enumerate}
        \item  For all $x \in V(G)$, $B_t(x) \cap S \neq \emptyset$.
        \item  For all $x, y \in V(G)$ with $x \neq y$, we must have $B_t(x) \cap S \neq B_t(y) \cap S$.
    \end{enumerate}
\end{defn}

The first condition in the definition requires that $S$ be a \textit{dominating set}.  The second condition requires that each vertex's \textit{identifying set} (the sets $B_t(x) \cap S$ and $B_t(y) \cap S$) is unique.  To settle the question of existence of $t$-identifying codes in a graph, we will rely on the following fact.  If a $t$-identifying code exists in a graph $G$, then we say that $G$ is $t$-identifiable.  If the variable $t$ is omitted, then we may assume that $t=1$.

\begin{defn}
    Two vertices $x,y$ are $t$-twins if $B_t(x)=B_t(y)$.
\end{defn}

\begin{fact}
    A graph is $t$-identifiable if and only if it does not contain any $t$-twins.
\end{fact}

Next we will define the class of de Bruijn graphs.  A good reference for the de Bruijn graphs and some of their properties is \cite{Baker}.  First, we define $[d]=\{0,1,2, \ldots, d-1\}$ (note that this definition is non-standard).  Then we define the de Bruijn graph as follows.

\begin{defn}
    Define the set $\mathcal{S}(d,n)$ to be the set of all strings of length $n$ over the alphabet $[d]$.  The directed de Bruijn graph $\vec{\mathcal{B}}(d,n)$ is the graph with vertex set $V = \mathcal{S}(d,n)$, and edge set $E = \mathcal{S}(d, n+1)$.  An edge $x_1 x_2 \ldots x_{n+1}$ denotes the edge from vertex $x_1 x_2 \ldots x_n$ to vertex $x_2 x_3 \ldots x_{n+1}$.  The undirected de Bruijn graph $\mathcal{B}(d,n)$ is $\vec{\mathcal{B}}(d,n)$ with undirected edges.
\end{defn}

Identifying codes were first introduced and defined in \cite{first}.  They have many interesting applications, such as efficiently placing smoke detectors in a house to provide maximum location information.  They are related to (but different from) dominating sets, perfect dominating sets, locating dominating sets, and many more types of vertex subsets.  In general, the problem of finding an identifying code in a graph is an NP-complete problem \cite{NPHard}.  Results on the existence and construction of identifying codes on the directed de Bruijn graph can be found in \cite{BH}.

In Section \ref{NBSection} we prove results on existence in $\mathcal{B}(d,n)$ for $d > 2$, while Section \ref{BSection} considers the case when $d=2$ separately.  Finally, we conclude with some open problems.

\section{Non-Binary de Bruijn Graphs}\label{NBSection}

We begin with our main result.

\begin{thm}\label{LargeN}
    $\mathcal{B}(d,n)$ is $t$-identifiable for $d \geq 3$ and $n \geq 2t$.
\end{thm}

To prove this theorem, we will use several lemmas that we prove first.  The first lemma (from \cite{LiuPaths}) is stated using our own terminology and with our own discussion/proof.

\begin{lem}\label{ball}
    The strings in $B_t(x)$ for $x = x_1x_2 \ldots x_n$ must be in one of the following three sets.
\begin{enumerate}
    \item  $\{x\}$;
    \item  $[d]^g \oplus x_{b-f+1} \ldots x_{n-f} \oplus [d]^{b-g}$ with $b > f, b>g, f+b+g \leq t$;
    \item  $[d]^{f-c} \oplus x_{b+1} \ldots x_{n-f+b} \oplus [d]^c$ with $f > b, f>c, b+f+c \leq t$.
\end{enumerate}
\end{lem}
\begin{proof}
    All strings in $B_t(x)$ can be described by following forward or backward edges.  The strings of type (1) are reached by taking no moves.  All other strings (types (2) and (3)) are reached by taking either moves of type FBF (forward-backward-forward) or BFB (backward-forward-backward).  We will describe \textit{shortest} paths within these confines.  We define $f$ steps forward from vertex $x_1x_2 \ldots x_n$ as reaching vertices in the set: $$[d]^f \oplus x_1 \ldots x_{n-f}.$$  We define $b$ steps backward from vertex $x_1x_2 \ldots x_n$ as reaching vertices in the set : $$x_{b+1} \ldots x_n \oplus [d]^b.$$

If FBF is the shortest path to reach some vertex $y$ from $x$, then we must follow $f$ edges forward, $b$ edges backward, and $g$ edges forward, with the constraints that $b>f$, $b>g$, and $f+b+g \leq t$. Following these sequences, we arrive at strings of type (2).

If BFB is the shortest path to reach some vertex $y$ from $x$, then we must follow $b$ edges backward, $f$ edges forward, and $c$ edges backward, with the constraints that $f>b$, $f>c$, and $b+f+c \leq t$. Following these sequences, we arrive at strings of type (3).
\end{proof}

Next, we will look at the possible $t$-prefixes that can appear in a special subset of $B_t(y)$.  A $t$-prefix of a string $x_1x_2 \ldots x_n$ is simply the first $t$ letters:  $x_1x_2 \ldots x_t$.  Since $[d]^t \oplus y_1y_2 \ldots y_{n-t} \subseteq B_t(y)$, if we consider the whole set $B_t(y)$ then every possible $t$-prefix must appear.  Instead, we want to determine an upper-bound on the number of distinct $t$-prefixes in $B_t(y) \setminus \left( [d]^t \oplus y_1y_2 \ldots y_{n-t} \right)$.  Eventually, we will show that this number of $t$-prefixes is smaller than $d^t$, so we will always be able to choose a $t$-prefix outside of this special subset.

\begin{lem}\label{counting}
    For $n \geq 2t$, the number of distinct $t$-prefixes in $B_t(y) \setminus [d]^t \oplus y_1y_2 \ldots y_{n-t}$ is at most  $$1-d^{\lfloor{ t/2\rfloor}} + 2 \cdot \sum_{j=0}^{t-1} d^j.$$
\end{lem}
\begin{proof}
    Following Lemma \ref{ball}, the $t$-prefixes in $B_t(y)$ take one of the following three forms (matching the types in Lemma \ref{ball}).
\begin{enumerate}
    \item  $y_1y_2 \ldots y_t$;
    \item  $[d]^g \oplus y_{b-f+1} \ldots y_{t+b-f-g}$;
    \item  $[d]^{f-c} \oplus y_{b+1} \ldots y_{t+b+c-f}$.
\end{enumerate}

In order to more easily count these $t$-prefixes, we will sort them by the last letter that appears in the $t$-prefix, and then sort them from longest $[d]^i$ prefix to smallest.  Since the largest $[d]^i$ prefix also counts the strings with smaller $[d]^j$ prefix so long as the strings end in the same letter, this will allow us to count unique prefixes.  We begin by rewriting the types of prefixes so as to more easily do this.

\begin{enumerate}
    \item  $y_1y_2 \ldots y_t$;
    \item  Recall the initial requirements for $b,f,g$ from Lemma \ref{ball}.  We find the range of $y$-subsequences by noticing that $b \geq g+1$, $f \leq t - b - g \leq t-2g-1$, and also that $b-f$ is maximized whenever $f=0$. If $f=0$, then we have either $b=t-g$, or if $g$ is large enough (i.e. $g=(t-1)/2$) we have $b=g+1$.  Combined, this gives us the following equations.
\begin{eqnarray*}
    \min (b-f) & = & (g+1)-(t-2g-1) \\
    & = & 3g+2-t, \hbox{ and } \\
    \max (b-f) & = & \max ( g+1, t-g) \\
    & = & t-g.
\end{eqnarray*}

Hence for $0 \leq g \leq \frac{t-1}{2}$:
\begin{eqnarray*}
    {[d]}^g & \oplus & y_{3g+2-t+1} \ldots y_{2g+2} \\
    & \vdots & \\
    {[d]}^g & \oplus & y_{t-g+1} \ldots y_{2t-2g}
\end{eqnarray*}

\noindent  Now we consider all of the possible last letters that might appear.

Last letters:  $y_i$ such that $2g+2 \leq i \leq 2t-2g$.

Range:  $y_i$ is a last letter whenever $t+1 \leq i \leq 2t$.

\noindent  So as to minimize the amount of double-counting, we index each of these $y_i$'s that appear by the choice of $g$ that forces it to appear last.

Max $g$ for each $y_i$:  $\lfloor \frac{2t-i}{2} \rfloor$.
    \item  Note that in this case, we can cover all cases with $c > 0$ by a different case with $c=0$, so we may just consider the cases $c=0$ to simplify things.  This is simply because if $c>0$, we may take $f'=f-c$, $c'=0$ to obtain the same $t$-prefix with smaller choices of $f,b,c$.  We use the same process as in (2) to determine the possible last letters and index them to minimize double-counting.
\begin{enumerate}
    \item  For $0 \leq f \leq \frac{t+1}{2}$:
\begin{eqnarray*}
    {[d]}^f & \oplus & y_1 \ldots y_{t-f} \\
    & \vdots & \\
    {[d]}^f & \oplus & y_f \ldots y_{t-1}
\end{eqnarray*}
Last letters:  $y_{\frac{t-1}{2}}, \ldots, y_{t-1}$.

Range:  $y_i$ is a last letter whenever $t-f \leq i \leq t-1$.

Max $f$ for each $y_i$:  $\frac{t+1}{2}$.

    \item  For $ \frac{t+1}{2} < f < t$ (recall we eliminated $f=t$):
\begin{eqnarray*}
    {[d]}^f & \oplus & y_1 \ldots y_{t-f} \\
    & \vdots & \\
    {[d]}^f & \oplus & y_{t-f+1} \ldots y_{2t-2f}
\end{eqnarray*}
Last letters:  $y_1, y_2, \ldots , y_{t-2}$.

Range:  $y_i$ is a last letter whenever $t-f \leq i \leq 2t-2f$.

Max $f$ for each $i$:  $\frac{2t-i}{2}$.
\end{enumerate}
\end{enumerate}

Note that because we require $n \geq 2t$, both cases (2) and (3) cover all possible $t$-prefixes.  That is, we cannot possibly have any $t$-prefixes that end in $[d]^k$ for any $k > 0$.  Additionally, note that each case covers a different range of last letters:  (1) $i=t$; (2) $t+1 \leq i \leq 2t$; and (3) $t-f \leq i \leq t-1$.  Hence we may count each case separately.

\begin{enumerate}
    \item  There is only one string in this case.

    \item  We showed previously that $\max (g) = \lfloor \frac{2t-i}{2} \rfloor$.  Thus we have the following formula.
$$\left\{
  \begin{array}{ll}
    d^{\frac{t-1}{2}}+2 \cdot \sum_{j=0}^{\frac{t-3}{2}} d^j, & \hbox{if $t$ is odd;} \\
    2 \cdot \sum_{j=0}^{\frac{t-2}{2}} d^j, & \hbox{if $t$ is even.}
  \end{array}
\right.$$

    \item  In this case, our subcases (a) and (b) overlap.  We break up our ranges slightly differently this time to determine $\max (f)$.
    \begin{enumerate}
        \item  $1 \leq i < \frac{t-1}{2}$.

            In this range for $i$, we must be in the higher range for $f$, so we have $\max (f) = \lfloor \frac{2t-i}{2} \rfloor$.

        \item  $\frac{t-1}{2} \leq i \leq t-2$.

            Considering both ranges for $f$, we have the following maximum value for $f$, depending on $i$.   $$\max(f) = \max \left( \frac{t+1}{2}, \left\lfloor \frac{2t-i}{2} \right\rfloor \right) = \left\lfloor \frac{2t-i}{2} \right\rfloor$$

        \item  $i=t-1$.

            For this value of $i$, we must be in the lower range for $f$, and hence we have $\max (f) = \lfloor \frac{t+1}{2} \rfloor = \lfloor \frac{2t-i}{2} \rfloor$.
    \end{enumerate}

    Hence all cases (a)-(c) have $\max (f) = \lfloor \frac{2t-i}{2} \rfloor$.  Thus we have the following formula.

$$\left\{
  \begin{array}{ll}
    2 \cdot \sum_{j=\frac{t+1}{2}}^{t-1} d^j, & \hbox{if $t$ is odd;} \\
    -d^{\frac{t}{2}}+2 \cdot \sum_{j=\frac{t}{2}+1}^{t-1} d^j, & \hbox{if $t$ is even.}
  \end{array}
\right.$$

\end{enumerate}

Now when we combine all of our equations we get the following final count. $$1-d^{\lfloor \frac{t}{2} \rfloor}+2 \cdot \sum_{j=0}^{t-1}d^j$$

Note that this provides only an upper bound on our $t$-prefixes - if we have repeated letters than we may have double-counted.
\end{proof}

Now we are ready to prove our theorem.

\begin{proof}[Proof of Theorem \ref{LargeN}]
    Consider two arbitrary strings: $x=x_1x_2\ldots x_n$ and $y=y_1 y_2 \ldots y_n$.  We will show that these two strings cannot be $t$-twins by showing that $B_t(x)\setminus B_t(y) \neq \emptyset$.  This will be done in two cases:  $x_1x_2 \ldots x_{n-t} \neq y_1y_2 \ldots y_{n-t}$ and $x_{t+1}x_{t+2} \ldots x_n \neq y_{t+1}y_{t+2} \ldots y_n$.  Note that this covers all cases, since $x \neq y$ implies there is some $i \in [1,n]$ such that $x_i \neq y_i$.  Additionally, since $n \geq 2t$, we must have that $i \in [1,n-t] \cup [t+1,n]$.  Hence at least one of these two cases must be true.
\begin{enumerate}
    \item  $x_1x_2 \ldots x_{n-t} \neq y_1y_2 \ldots y_{n-t}$.

    We will show that there must exist some string in $B_t(x)$ that is not in $B_t(y)$.  In particular, there is a string $a \in [d]^t \oplus x_1 \ldots x_{n-t}$ such that $a \not \in B_t(y)$. We do this by counting the number of distinct $t$-prefixes in $B_t(y) \setminus [d]^t \oplus y_1y_2 \ldots y_{n-t}$, and showing that this number is smaller than $d^t$.  Note that because of the case that we are in, we need not consider the strings in $[d]^t \oplus y_1y_2 \ldots y_{n-t}$.  If we can show that the number of $t$-prefixes is smaller than $d^t$, then there must be some string $z \in B_t(x) \setminus B_t(y)$.

From Lemma \ref{counting}, we know that the total number of $t$-prefixes in $\mathcal{B}_t(y) \setminus [d]^t \oplus y_1y_2 \ldots y_{n-t}$ is equal to $1-d^{\lfloor \frac{t}{2} \rfloor}+2 \cdot \sum_{j=0}^{t-1}d^j$, and that one of those $t$-prefixes is $y_1 \ldots y_t$, which we may ignore because of the case that we are in.  Define $f(t) = -d^{\lfloor \frac{t}{2} \rfloor}+2 \cdot \sum_{j=0}^{t-1}d^j$ and $g(t) = d^t - f(t)$.  If we can show that $g(t)$ is always positive for $d \geq 3$, then we know that there exists a string $a \in \left( [d]^t \oplus x_1 \ldots x_{n-t} \right) \setminus \left( [d]^t \oplus y_1 \ldots y_{n-t} \right) \subseteq B_t(x) \setminus B_t(y)$. Then we know that $x$ and $y$ are not $t$-twins.

Consider our new function $g(t)$.
\begin{eqnarray*}
    g(t) & = & d^t + d^{t/2} - 2 \cdot \sum_{j=0}^{t-1} d^j \\
    & = & d^t + d^{t/2} - \frac{2 \cdot (d^t-1)}{d-1} \\
    & = & \frac{d^t(d-1)+d^{t/2}(d-1)-2(d^t-1)}{d-1}
\end{eqnarray*}
We will determine the nature of this function by finding the roots.  We find the roots by setting the numerator equal to 0 and making a substitution $x = d^{t/2}$.
\begin{eqnarray*}
    d^t(d-1)+d^{t/2}(d-1)-2(d^t-1) & = & x^2(d-3)+x(d-1)+2
\end{eqnarray*}
    The roots of this equation are $x=-1$ and $x = \frac{-4}{2d-6}$.  Reversing our substitution this equates to $d^{t/2} = -1$ and $d^{t/2} = \frac{-4}{2d-6}$.  The first root is impossible, and the second will only be possible when $2d-6 < 0$, or $d<3$.  Hence, if $d \geq 3$, our function has no real roots and is always positive.
    \item  $x_{t+1}x_{t+2} \ldots x_n \neq y_{t+1}y_{t+2} \ldots y_n$.

    In this case, we want to show that there exists some string:
$$a \in \left(x_{t+1} \ldots x_n \oplus [d]^t\right) \setminus \left(y_{t+1} \ldots y_n \oplus [d]^t \right) \subseteq B_t(x) \setminus B_t(y).$$
Because of the symmetric nature of the strings and edges in the de Bruijn graph, this case follows the same as the previous case, with analogous lemmas to Lemmas \ref{ball} and \ref{counting} for $t$-suffixes (instead of $t$-prefixes).  Thus we will again always have fewer than $d^t$ prefixes represented in $B_t(y) \setminus \left(y_{t+1} \ldots y_n \oplus [d]^t \right)$, so we will always be able to find the desired string $a$ that can identify $x$ from $y$.
\end{enumerate}
\end{proof}

%
%
%

As a separate result, we show that $\mathcal{B}(d,3)$ is $2$-identifiable for $d \geq 3$.

\begin{thm}
    $\mathcal{B}(d,3)$ is 2-identifiable whenever $d \geq 3$.
\end{thm}
\begin{proof}
    Let $x = x_1x_2x_3$ and $y = y_1y_2y_3$ be distinct vertices in $\mathcal{B}(d,3)$.  We consider three cases.
    \begin{description}
        \item[Case 1]  $x_1 \neq y_1$.

        Let $a_1a_2a_3 \in [d] \oplus [d] \oplus x_1$ such that $a_1a_2$ is not one of the following strings or is not contained in one of the sets of strings.
        $$\begin{array}{c}
            [d] \oplus y_1 \\
            y_2y_3 \\
            y_3 \oplus [d] \\
            y_1y_2 \\
            \left[d\right] \oplus y_2
        \end{array}$$
        We have a total of $[d]^2$ options for $a_1a_2$, and this list contains at most $3d-1$ of those choices.  Hence for $d \geq 3$, there is always an option left for $a_1a_2$.  Then we have $a_1a_2a_3 \in B_2(x) \setminus B_2(y)$.
        \item[Case 2]  $x_3 \neq y_3$.

        Let $a_1a_2a_3 \in x_3 \oplus [d] \oplus [d]$ such that $a_2a_3$ is not one of the following strings or is not contained in one of the sets of strings.
        $$\begin{array}{c}
            y_1y_2 \\
            y_3 \oplus [d] \\
            \left[d\right] \oplus y_1 \\
            y_2 \oplus [d] \\
            y_2y_3
        \end{array}$$
        Note that this set of strings has at most $3d-1$ elements, and hence we can always find some choice for $a_2a_3$ that is allowed.  Then we have $a_1a_2a_3 \in B_2(x) \setminus B_2(y)$.
        \item[Case 3]  $x_2 \neq y_2$ and $x_1x_3 = y_1y_3$.

        We break this case up further into three subcases.
        \begin{enumerate}
            \item  If $x_1x_2 = y_2y_3$, then we must have $x=abb$ and $y=aab$ for some $a \neq b$.  Then $cbb \in B_2(x) \setminus B_2(y)$ for any choice of $c \in [d] \setminus \{a,b\}$.
            \item  If $x_1=x_3$, then $x=aba$ and $y=aca$ for some $b \neq c \in [d]$.  We must have either $b \neq a$ or $c \neq a$, and so without loss of generality we may assume $b \neq a$.  Then $kab \in B_2(x) \setminus B_2(y)$ for any choice of $k \in [d] \setminus \{a,c\}$.
            \item  Lastly, if we are not in either of the two previous subcases then we choose $a_1a_2a_3 \in x_1x_2 \oplus [d]$ with $a_3 \in [d] \setminus \{y_1,y_2\}$.  Then we have $a_1a_2a_3 \in B_2(x) \setminus B_2(y)$.
        \end{enumerate}
    \end{description}
\end{proof}

For the remaining cases where $n < 2t$, a different argument must be found.  While this problem remains open, we believe that the following result could be useful in solving these cases.

\begin{thm}\label{NDist}
    For any $y \in \mathcal{B}(d,n)$ with $d \geq 3$, there exists some vertex $x$ such that $d(y,x)=n$.
\end{thm}

\begin{proof}
    We proceed by induction on $n$ and show that if the claim is true in $\mathcal{B}(d,n)$ for $n \geq 2$, then the claim is true for $\mathcal{B}(d,n+2)$.
    \begin{description}
        \item[Base Case:]  $n=2$.  Since $d \geq 3$, our vertex $y=y_1y_2$ can use at most two symbols from our alphabet.  Suppose that $z \in [d] \setminus \{y_1,y_2\}$.  Then $d(y, zz)=2$.

            As our induction proceeds from string length $n$ to $n+2$, we require an additional base case of $n=3$.  If our vertex $y=y_1y_2y_3$ only uses two distinct symbols from $[d]$, then the string $x = a^n$ where $a \in[d] \setminus \{y_1, y_2,y_3\}$ satisfies $d(y,x)=3$.  Otherwise, we must have $[d]=\{y_1,y_2,y_3\}$.  Then the vertex $x=(y_2)^3$ satisfies $d(y,x)=3$.

        \item[Induction Step:]  Let $\overline{y}=y_0 \oplus y \oplus y_{n+1}$ be arbitrary. By the induction hypothesis, there exists some $x \in \mathcal{B}(d,n)$ such that $d(x,y)=n$.  We will show that $d(\overline{y}, \overline{x})=n+2$, where $\overline{x}=x_0 \oplus x \oplus x_{n+1}$ with $x_0 \in [d] \setminus \{y_n,y_{n+1}\}$ and $x_{n+1} \in [d] \setminus \{y_0, y_1\}$.  We will show that $\overline{x} \not \in B_{n+1}(y)$ using Lemma \ref{ball} and considering each type of path and resulting string individually.
        \begin{enumerate}
            \item  $\overline{x}=\overline{y}$.  Not possible since $x \neq y$.
            \item  FBF-type.

                First, from Lemma \ref{ball}, we know that since $d(x,y)=n$ there cannot exist any choice of $f,b,g$ such that $f+b+g \leq n-1$, $b>0$, $b > f$, and $b>g$ such that $$x \in [d]^g \oplus y_{b-f+1} \ldots y_{n-f} \oplus [d]^{b-g}.$$  In other words, we must have $$y_{b-f+1} \ldots y_{n-f} \neq x_{g+1} \ldots x_{g+n-b}$$ for all such choices of $f,b,g$.

                Now we will show that there does not exist an FBF-path of length $n+1$ or less between $\overline{x}$ and $\overline{y}$.  Fix some $f,b,g$ such that $f+b+g \leq n+1$, $b>0$, $b>f$, and $b>g$.  From Lemma \ref{ball} all vertices $z_0z_1 \ldots z_{n+1}$ that can be reached by an FBF-path with parameters $(f,b,g)$ from $\overline{y}$ must have $$y_{b-f} \ldots y_{n-f+1} = z_{g} \ldots z_{g+n+1-b}.$$

            \begin{enumerate}
                \item  If $f=0$, $b=k$, and $g=0$, then we consider $1 \leq k \leq n-1$ and $n \leq k \leq n+1$ separately.  First, if $1 \leq k \leq n-1$, then our induction hypothesis with parameters $(0,k,0)$ tells us that $x_1 \ldots x_{n-k} \neq y_{k+1} \ldots y_n$ when we examine FBF-paths with parameters $(0,k,0)$ from $y$.  Hence we cannot have $x_0 \ldots x_{n-k+1} = y_k \ldots y_{n+1}$, and so no such FBF-path exists between $\overline{x}$ and $\overline{y}$.  Next, if $n \leq k \leq n+1$, then since $x_0 \neq y_n, y_{n+1}$, we will never have $x_0 x_1=y_n y_{n+1}$ or $x_0=y_{n+1}$, and so again no such FBF-path exists in $\mathcal{B}(d,n+2)$.
                \item  If $f \geq 1$, then we must have $b \geq 2$.  In this case, in order for such an FBF-path to exists from $\overline{y}$ to $\overline{x}$ we must have $x_g \ldots x_{g+n-b+1} = y_{b-f} \ldots y_{n+1-f}$.  However our induction hypothesis with parameters $(f-1, b-1, g)$ tells us that $x_{g+1} \ldots x_{g+n-b+1} \neq y_{b-f+1} \ldots y_{n-f+1}$, and so no such FBF-path exists in $\mathcal{B}(d,n+2)$.
                \item  If $g \geq 1$, then again we must have $b \geq 2$.  In this case, in order for such an FBF-path to exist we must have $x_g \ldots x_{g+n-b+1} = y_{b-f} \ldots y_{n+1-f}$.  However our induction hypothesis with parameters $(f, b-1, g-1)$ tells us that $x_g \ldots x_{g+n-b} \neq y_{b-f} \ldots y_{n-f}$, and so no such FBF-path exists in $\mathcal{B}(d,n+2)$.
            \end{enumerate}

            Hence we cannot have an FBF-path of length less than $n+2$ between $\overline{y}$ and $\overline{x}$ in $\mathcal{B}(d,n+2)$.
            \item  BFB-type.

First, from Lemma \ref{ball}, we know that since $d(x,y)=n$ there cannot exist any choice of $b,f,c$ such that $b+f+c \leq n-1$, $f>0$, $f > b$, and $f>c$ such that $$x \in [d]^{f-c} \oplus y_{b+1} \ldots y_{n-f+b} \oplus [d]^{c}.$$  In other words, we must have $$y_{b+1} \ldots y_{n-f+b} \neq x_{f-c+1} \ldots x_{n-c}$$ for all such choices of $b,f,c$.

Now we will show that there does not exist a BFB-path of length $n+1$ or less between $\overline{x}$ and $\overline{y}$.  Fix some $b,f,c$ such that $b+f+c \leq n+1$, $f>0$, $f>b$, and $f>c$.  From Lemma \ref{ball} all vertices $z_0z_1 \ldots z_{n+1}$ that can be reached by a BFB path from $\overline{y}$ with these parameters must have $$y_{b} \ldots y_{n+1-f+b} = z_{f-c} \ldots z_{n+1-c}.$$

            \begin{enumerate}
                \item  If $b=0$, $f=k$, and $c=0$, then we consider $1 \leq k \leq n-1$ and $n \leq k \leq n+1$ separately.  First, if $1 \leq k \leq n-1$, then our induction hypothesis tells us that $x_{k+1} \ldots x_{n} \neq y_{1} \ldots y_{n-k}$ when we examine BFB-paths with parameters $(0,k,0)$ from $y$.  Hence we cannot have $x_k \ldots x_{n+1} = y_0 \ldots y_{n-k+1}$ in $\mathcal{B}(d,n+2)$, so no such BFB-path exists between $\overline{x}$ and $\overline{y}$.

                    Next, if $n \leq k \leq n+1$, then since $x_{n+1} \neq y_0, y_{1}$, we will never have $x_nx_{n+1}=y_0y_1$ or $x_{n+1}=y_0$, and so again no such BFB-path exists in $\mathcal{B}(d,n+2)$.
                \item  If $b \geq 1$, then we must have $f \geq 2$.  In this case, in order for such a BFB-path to exist from  $\overline{x}$ to $\overline{y}$ we must have $x_{f-c} \ldots x_{n+1-c} = y_{b} \ldots y_{n+1-f+b}$.  However our induction hypothesis with parameters $(b-1, f-1, c)$ tells us that $$x_{f-c} \ldots x_{n-c} \neq y_{b} \ldots y_{n-f+b},$$ and so no such BFB-path exists in $\mathcal{B}(d,n+2)$.
                \item  If $c \geq 1$, then again we must have $f \geq 2$.  In this case, in order to have such a BFB-path between $\overline{x}$ and $\overline{y}$ we must have $x_{f-c} \ldots x_{n+1-c} = y_{b} \ldots y_{n+1-f+b}$.  However our induction hypothesis with parameters $(b, f-1, c-1)$ tells us that $x_{f-c+1} \ldots x_{n-c+1} \neq y_{b+1} \ldots y_{n-f+1+b}$, and so no such BFB-path exists in $\mathcal{B}(d,n+2)$.
            \end{enumerate}

    Hence we cannot have a BFB-path of length less than $n+2$ between $\overline{y}$ and $\overline{x}$ in $\mathcal{B}(d,n+2)$.
        \end{enumerate}

    Therefore there is no path from $\overline{y}$ to $\overline{x}$ of length $n+1$ or smaller, and so $d(\overline{y},\overline{x})\geq n+2$.  As it is well known that the de Bruijn graph $\mathcal{B}(d,n+2)$ has diameter $n+2$ (see \cite{Baker}), we must have $d(\overline{y},\overline{x})=n+2$.
    \end{description}
\end{proof}

In other words, Theorem \ref{NDist} tells us the eccentricity of every node in the graph $\mathcal{B}(d,n)$ is $n$ for $d \geq 3$, and so the radius of $\mathcal{B}(d,n)$ is $n$.  Note that when $d=2$ this does not always hold.  For example, the graph $\mathcal{B}(2,3)$ does not have any vertex at distance 3 from $011$.  See Figure \ref{B23Ex}.

\begin{figure}
\begin{center}

\begin{tikzpicture}[-,>=stealth',auto,node distance=2cm,
  thick,main node/.style={circle,draw,font=\sffamily\bfseries,scale=0.75},new node/.style={circle,fill=black,text=white,draw,font=\sffamily\bfseries,scale=0.75}]

  \node[main node] (0) {000};
  \node[main node]  (1) [above right of=0] {001};
  \node[main node] (2) [below right of=1] {010};
  \node[main node]  (4) [below right of=0] {100};
  \node[main node] (5) [right of=2]       {101};
  \node[main node]  (6) [below right of=5] {110};
  \node[new node]  (3) [above right of=5] {011};
  \node[main node] (7) [below right of=3] {111};

  \path[every node/.style={font=\sffamily\small}]
    (0) edge node [left]      {} (1)
        edge [loop left] node {} (0)
    (1) edge node [left]      {} (3)
        edge node [right]     {} (2)
    (2) edge [bend right] node{} (5)
        edge node [right]     {} (4)
    (3) edge node [right]     {} (6)
        edge node [right]     {} (7)
    (4) edge node [left]      {} (0)
        edge node [right]     {} (1)
    (5) edge [bend right] node{} (2)
        edge node [right]     {} (3)
    (6) edge node [right]     {} (5)
        edge node [right]     {} (4)
    (7) edge [loop right] node{} (7)
        edge node [right]     {} (6);

\end{tikzpicture}

\end{center}
\caption{$\mathcal{B}(2,3)$ does not contain any vertices at distance 3 from 011.}  \label{B23Ex}
\end{figure}
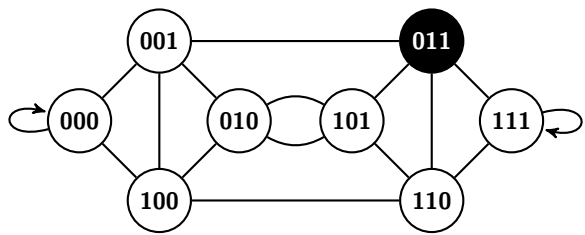

\section{Binary de Bruijn Graphs}\label{BSection}

We now consider the binary de Bruijn graphs.  We provide one result within this range, and show that $\mathcal{B}(2,n)$ is always 1-identifiable.

\begin{thm}
    For $n\geq 3$, the graph $\mathcal{B}(2,n)$ is identifiable.
\end{thm}
\begin{proof}
    For $n=3$, the following is a minimum 1-identifying code on $\mathcal{B}(2,3)$.  $$\{001, 010, 011, 101\}$$

    When $n \geq 4$, we have the following proof, with many cases.  We will prove this result by showing that it is not possible to have two vertices $x$ and $y$ that are twins.  Suppose (for a contradiction) that $x$ and $y$ are in fact twins in $\mathcal{B}(2,n)$.  First, the 1-balls for each vertex are as follows.
$$\begin{array}{rclcrcl}
    B_1(x) & = & \left\{  \begin{array}{l} x_1 x_2 \ldots x_n \\ 0 x_1 \ldots x_{n-1} \\ 1 x_1 \ldots x_{n-1} \\ x_2 \ldots x_n 0 \\ x_2 \ldots x_n 1  \end{array} \right\} & & B_1(y) & = & \left\{ \begin{array}{l} y_1 y_2 \ldots y_n \\
            0 y_1 \ldots y_{n-1} \\
            1 y_1 \ldots y_{n-1} \\
            y_2 \ldots y_n 0 \\
          y_2 \ldots y_n 1 \end{array} \right\}
\end{array}$$

Without loss of generality, we assume that $x_1 = 0$.  Then we have two cases:  either $x_1x_2 \ldots x_n = 0y_1 \ldots y_{n-1}$, or $x_1x_2 \ldots x_n \in \{y_2 \ldots y_n0, y_2 \ldots y_n 1\}$.
\begin{enumerate}
    \item  $x_1x_2 \ldots x_n = 0y_1 \ldots y_{n-1}$.

            In this case, we know that $0x_2 \ldots x_n = 0y_1 \ldots y_{n-1}$, and so $x_2 \ldots x_n = y_1 \ldots y_{n-1}$.  From this, we know the following equality holds.  $$\{x_2 \ldots x_n 0, x_2 \ldots x_n 1\} = \{y_1 y_2 \ldots y_n, y_1 y_2 \ldots \overline{y_n}\}$$  This gives us two cases:  either $y_1y_2 \ldots \overline{y_n} \in \{0y_1 \ldots y_{n-1}, 1y_1 \ldots y_{n-1}\}$, or $y_1y_2 \ldots \overline{y_n} \in \{y_2 \ldots y_n 0, y_2 \ldots y_n 1\}$.
\begin{enumerate}
    \item  $y_1y_2 \ldots \overline{y_n} \in \{0y_1 \ldots y_{n-1}, 1y_1 \ldots y_{n-1}\}$

            The fact that $y_2 \ldots \overline{y_n} = y_1 \ldots y_{n-1}$ implies the following. $$y_1 = y_2 = \cdots = y_{n-1} = \overline{y_n}$$  Because we are in Case 1 and $x_2 \ldots x_n = y_1 \ldots y_{n-1}$, we also have the following equalities. $$x_2 = x_3 = \cdots = x_n = y_1$$  Hence our 1-balls must be as shown below for some $a \in \{0,1\}$.
$$\begin{array}{rclcrcl}
    B_1(x) & = & \left\{  \begin{array}{l} 0 a \ldots a \\
0 0 a \ldots a \\
1 0 a \ldots a \\
a \ldots a 0 \\
a \ldots a 1  \end{array} \right\} & & B_1(y) & = & \left\{ \begin{array}{l} a \ldots a \overline{a} \\
            0 a \ldots a \\
            1 a \ldots a \\
            a \ldots a \overline{a} 0 \\
            a \ldots a \overline{a} 1 \end{array} \right\}
\end{array}$$

Note that since $n \geq 4$, we have two strings in $B_1(y)$ that have different second-to-last and third-to-last letters, however in $B_1(x)$ there are no such strings.  Hence these sets cannot possibly be equal, which is a contradiction.
    \item  $y_1y_2 \ldots \overline{y_n} \in \{y_2 \ldots y_n 0, y_2 \ldots y_n 1\}$

This implies that $y_1y_2 \ldots y_{n-1} = y_2 \ldots y_n$, and so we have the following chain of equalities. $$y_1=y_2 = \cdots = y_{n-1} = y_n$$  Hence $y=a^n$ and $x=0a^{n-1}$ for some $a \in \{0,1\}$.  Since $x \neq y$, we must have $a=1$ and thus our 1-balls, given below, are clearly not equal - a contradiction.

$$\begin{array}{rclcrcl}
    B_1(x) & = & \left\{  \begin{array}{l} 0 1 \ldots 1 \\
0 0 1 \ldots 1 \\
1 0 1 \ldots 1 \\
1 \ldots 1 0 \\
1 \ldots 1 1  \end{array} \right\} & & B_1(y) & = & \left\{ \begin{array}{l} 1 1 \ldots 1  \\
            0 1 \ldots 1 \\
            1 \ldots 1  0  \end{array} \right\}
\end{array}$$
\end{enumerate}
    \item  $x_1x_2 \ldots x_n \in \{y_2 \ldots y_n0, y_2 \ldots y_n 1\}$ and $y_2=0$.

    From this, we have the following 1-balls.
$$\begin{array}{rclcrcl}
    B_1(x) & = & \left\{  \begin{array}{l} 0x_2 \ldots x_n \\
0 0 x_2 \ldots x_{n-1} \\
1 0 x_2 \ldots x_{n-1} \\
x_2 \ldots x_n 0 \\
x_2 \ldots x_n 1  \end{array} \right\} & & B_1(y) & = & \left\{ \begin{array}{l} y_1 0 x_2 \ldots x_{n-1} \\
            0 y_1 0 x_2 \ldots x_{n-2} \\
            1 y_1 0 x_2 \ldots x_{n-2} \\
            0 x_2 \ldots x_{n-1} 0 \\
            0 x_2 \ldots x_{n-1} 1 \end{array} \right\}
\end{array}$$
Now we have two cases:  either $1y_10x_2\ldots x_{n-2}= 10x_2 \ldots x_{n-1}$, or $1y_10x_2\ldots x_{n-2} \in \{x_2 \ldots x_n 0, x_2 \ldots x_n 1\}$.
\begin{enumerate}
    \item  $1y_10x_2\ldots x_{n-2}= 10x_2 \ldots x_{n-1}$.

    This statement implies that we have the following chain of equalities. $$y_3 = \cdots = y_n = x_2 = \cdots = x_{n-1}$$  In particular, we now know that $x=0a \ldots a$ and $y=0 0 a \ldots a$.  Hence our 1-balls are given below.
$$\begin{array}{rclcrcl}
    B_1(x) & = & \left\{  \begin{array}{l} 0 a \ldots a \\
0 0 a \ldots a \\
1 0 a \ldots a \\
a \ldots a 0 \\
a \ldots a 1  \end{array} \right\} & & B_1(y) & = & \left\{ \begin{array}{l} 0 0 a \ldots a \\
            0 0 0 a \ldots a \\
            1 0 0 a \ldots a \\
            0 a \ldots a 0 \\
            0 a \ldots a 1 \end{array} \right\}
\end{array}$$
Since $000a \ldots a \in B_1(y)$, the only way to have $B_1(x)=B_1(y)$ would require $a=0$, and thus $x =y$, which is a contradiction.
    \item  $1y_10x_2\ldots x_{n-2} \in \{x_2 \ldots x_n 0, x_2 \ldots x_n 1\}$ and $x_2=1$.

In this instance, we know that $x_2 \ldots x_n = 1 y_1 0 x_2 \ldots x_{n-3}$, and hence $x_5 \ldots x_n = x_2 \ldots x_{n-3}$.  This tells us that $x=01y_101y_1 \ldots $ and $y= y_101y_101\ldots $.  In particular, our 1-balls are now shown below.
$$\begin{array}{rclrcl}
    B_1(x) & = & \left\{  \begin{array}{l} 01y_101y_1 \ldots \\
0 01y_101y_1 \ldots \\
1 01y_101y_1 \ldots \\
1y_101y_1 \ldots 0 \\
1y_101y_1 \ldots 1  \end{array} \right\} &  B_1(y) & = & \left\{ \begin{array}{l} y_101y_101\ldots \\
            0 y_101y_101\ldots \\
            1 y_101y_101\ldots \\
            01y_101\ldots 0 \\
            01y_101\ldots 1 \end{array} \right\}
\end{array}$$
Note that $B_1(y)$ contains two distinct strings beginning with $01$, while $B_1(x)$ contains only one such string.  Hence it is not possible that $B_1(x)=B_1(y)$, which contradicts our initial assumption.
\end{enumerate}
\end{enumerate}
\end{proof}

Due to the fact that the eccentricity in the binary de Bruijn graph $\mathcal{B}(2,n)$ is not always equal to $n$, we know that there will be cases when a $t$-identifying code does not exist.

\section{Future Work}

We have the following questions to consider.

\begin{enumerate}
    \item  Is there a pattern for when $\mathcal{B}(2,n)$ is $t$-identifiable?  A related question is to determine the eccentricity for the undirected binary de Bruijn graph.
    \item  Can we determine when $\mathcal{B}(d,n)$ is $t$-identifiable for $d \geq 3$ and $n < 2t$?  Computer testing has led us to conjecture that for $d \geq 3$ and $n \geq 2$, there exists a $t$-identifying code in $\mathcal{B}(d,n)$ for $1 \leq t \leq n-1$.
    \item  What is the minimum possible size for an identifying code in $\mathcal{B}(d,n)$?  Are there any efficient constructions for either optimal or non-optimal identifying codes in these graphs?
\end{enumerate}

\bibliographystyle{amsplain}

\end{document}